\documentclass{amsart}

\usepackage[utf8]{inputenc}
\usepackage{amsfonts}
\usepackage{amsmath}
\usepackage{amsthm}
\usepackage{amssymb}
\usepackage{latexsym}
\usepackage{enumerate}
\usepackage{tikz}

\newtheorem{theorem}{Theorem}

\newtheorem{lemma}[theorem]{Lemma}
\newtheorem{proposition}[theorem]{Proposition}
\newtheorem{corollary}[theorem]{Corollary}

\newtheorem{question}[theorem]{Question}

\newtheorem{definition}[theorem]{Definition}

\newtheoremstyle{TheoremNum}
  {\topsep}{\topsep}              
  {\itshape}                      
  {}                              
  {\bfseries}                     
  {.}                             
  { }                             
  {\thmname{#1}\thmnote{ \bfseries #3}}
\theoremstyle{TheoremNum}
\newtheorem{reptheorem}{Theorem}

\newtheoremstyle{TheoremNum}
  {\topsep}{\topsep}              
  {\itshape}                      
  {}                              
  {\bfseries}                     
  {.}                             
  { }                             
  {\thmname{#1}\thmnote{ \bfseries #3}}
\theoremstyle{TheoremNum}

\begin{document}
\newcommand{\F}{\mathcal{F}}
\newcommand{\A}{\mathcal{A}}
\newcommand{\add}{\mathfrak{add}}
\newcommand{\cov}{\mathfrak{cov}}
\newcommand{\unif}{\mathfrak{unif}}
\newcommand{\cof}{\mathfrak{cof}}
\newcommand{\uu}{\mathfrak{u}}
\newcommand{\cc}{\mathfrak{c}}
\newcommand{\kk}{\mathfrak{k}}
\newcommand{\dd}{\mathfrak{d}}
\newcommand{\sss}{\mathfrak{s}}
\newcommand{\aaa}{\mathfrak{a}}
\newcommand{\pp}{\mathfrak{p}}
\newcommand{\ttt}{\mathfrak{t}}
\newcommand{\mm}{\mathfrak{m}\mathfrak{m}}
\newcommand{\bb}{\mathfrak{b}}
\newcommand{\ii}{\mathfrak{i}}
\newcommand{\N}{\mathbb{N}}
\newcommand{\M}{\mathbb{M}}
\newcommand{\Q}{\mathbb{Q}}
\newcommand{\R}{\mathbb{R}}
\newcommand{\K}{\mathbb{K}}
\newcommand{\C}{\mathbb{C}}
\newcommand{\PP}{\mathbb{P}}
\newcommand{\forces}{\Vdash}
\newcommand{\sgn}{\operatorname{sgn}}

\title[Rosenthal families,
 pavings  and  generic cardinal invariants]{ Rosenthal families,
 pavings  and  generic cardinal invariants}
\author{Piotr Koszmider}
\address{Institute of Mathematics, Polish Academy of Sciences,
ul. \'Sniadeckich 8,  00-656 Warszawa, Poland}
\email{\texttt{piotr.koszmider@impan.pl}}

\author{Arturo Mart\'\i nez-Celis}
\address{Institute of Mathematics, Polish Academy of Sciences,
ul. \'Sniadeckich 8,  00-656 Warszawa, Poland}
\email{\texttt{arodriguez@impan.pl}}
%
%

%
\begin{abstract} Following D. Sobota we call a family $\mathcal F$ of infinite subsets
of $\N$ a Rosenthal family if it can replace the family of all infinite subsets of $\N$
in  classical Rosenthal's Lemma concerning sequences of measures on pairwise disjoint sets.
We resolve two problems on Rosenthal families: every ultrafilter is a Rosenthal family and
the minimal size of a Rosenthal family  is exactly equal to the reaping cardinal 
$\mathfrak r$.
This is achieved  through analyzing nowhere reaping families of subsets of $\N$
and through applying a paving lemma which is a consequence of a paving lemma
concerning linear operators on $\ell_1^n$ due to Bourgain. 
We use connections of the above results with free set results for 
 functions on $\N$ and with linear operators on $c_0$ to  determine the values of several other 
derived cardinal invariants.

\end{abstract}

\maketitle

\section{Introduction}

Recall that the {\sl reaping number} $\mathfrak r$  is the minimal cardinality of
a family $\mathcal F$ of infinite subsets of $\N$ (denoted by $[\N]^\omega$) 
which is not split by a single subset of $\N$, i.e.,
such that there is no $A\subseteq \N$ such that $A\cap F$ and $F\setminus A$ are both
infinite for all $F\in \mathcal F$ (see \cite{blass} for more details).

A family $\mathcal D$ of infinite subsets of $\N$ 
 will be called {\sl dense} if
for every infinite $A\subseteq \N$ there is $B\in \mathcal D$ such that $B\subseteq A$.
Let $\mathsf D$   be a collection of dense sets. A family $\mathcal G\subseteq[\N]^\omega$  is called
a {\sl generic family} for $\mathsf D$ if $\mathcal G\cap \mathcal D\not=\emptyset$
for all $\mathcal D\in \mathsf D$. This terminology
agrees with the standard one for the partial order $([\N]^\omega, \subseteq)$ but
it should be stressed that our generic families need not to be filters. 
We define {\sl the generic cardinal number
 $\mathfrak{gen}(\mathsf D)$ for} $\mathsf D$
to be the minimal cardinality of a generic family for $\mathsf D$. It is clear
that if $\mathsf D$ is as above, then $\mathfrak{gen}(\mathsf D)\leq\mathfrak c$
as $[\N]^\omega$ is a generic family for $\mathsf D$.

As an example consider a sequence $(x_n)_{n\in \N}$ of zeros and ones. Let 
${{Conv}}^{01}_{(x_n)_{n\in \N}}$ 
be the family of all $A\in[ \N]^\omega$ such that $(x_n)_{n\in A}$ converges and
let $\mathsf{Conv}_{01}$ be the collection of all such sets ${Conv}^{01}_{(x_n)_{n\in \N}}$.
Then $\mathfrak{gen}(\mathsf{Conv}_{01})=\mathfrak{r}$ (Theorem 3.7 of \cite{blass}).
Here the Bolzano-Weierstrass theorem plays the role of a {\sl density lemma} i.e.,
a result asserting that the families in question are dense.
If $(x_n)_{n\in \N}$ above is an arbitrary bounded 
sequence of reals and  $\mathsf{Conv}$ is the family of all analogous dense
sets ${Conv}_{(x_n)_{n\in \N}}$, then  $\mathfrak{gen}(\mathsf{Conv})=\mathfrak{r}_\sigma$,
where $\mathfrak{r}_\sigma$ is a modified version of $\mathfrak r$ (see Section 3 of \cite{blass}).

Another  example of a density lemma, dense sets, generic families and the generic cardinal invariant is the following: Let $f: [\N]^2\rightarrow \{0,1\}$ and let
${Hom}_f$  be the family of all infinite subsets of $\N$ which are homogeneous for $f$. 
The Ramsey theorem as a density lemma yields the density of each set ${Hom}_f$.
If $\mathsf{Hom}$ is the collection of all such families ${Hom}_f$ 
for all functions $f$ as above, then
$\mathfrak{gen}(\mathsf{Hom})=\max(\mathfrak{d}, 
\mathfrak{r}_\sigma)$ (Theorem  3.10 of \cite{blass}), where 
the dominating number $\mathfrak d$ is the minimal size
of a family of functions from $\N$ to $\N$ which eventually dominate any such function. 

Moreover, it is proved in 4.7 and 4.9 of \cite{booth} that an ultrafilter is generic for
 $\mathsf{Conv}$ if and only if it is p-point and it is generic for
$\mathsf{Hom}$ if and only if it is  selective.
In particular, by a theorem of S. Shelah  it is consistent that there are no generic ultrafilters for
 $ \mathsf{Conv}$ and $\mathsf{Hom}$ (\cite{wimmers}, cf.  \cite{ppoints}). 
 This,  means that there are no generic families  
 for these collections of dense sets which satisfy the strong finite intersection property
 (i.e., intersection of every finite subfamily is infinite) as any ultrafilter extending such families would need to be a p-point or a selective  ultrafilter by the results of \cite{booth}.

The topic of this paper falls into the category of the results described above. 
Originally we were motivated by a classical density lemma  frequently used in several parts of
mathematics 
concerning sequences of measures (for the discussion of
its classical forms and uses see Subsection \ref{measures}) which can
be stated in the following  equivalent combinatorial form:

\begin{lemma}[Rosenthal's lemma; \cite{rosenthal1}, \cite{rosenthal2}]\label{rosenthal-intro}
Suppose that $M=(m_{k, n})_{k, n\in \N}$ is a matrix 
of non-negative reals, where  the set of sums $\{\sum_{n\in \N}m_{k, n}: k\in \N\}$
 is bounded.  For every $\varepsilon>0$ and every infinite $B\subseteq \N$  there is an 
 infinite $A\subseteq B$ 
such that for every $k\in A$ we have
$$\sum_{n\in A\setminus \{k\}}m_{k, n}\leq \varepsilon.\leqno (*)$$
\end{lemma}

Matrices $M$ as above will be called {\sl Rosenthal matrices}, the set of
all of them will be denoted  by $\mathbb M$. The supremum
of the set $\{\sum_{n\in \N}m_{k, n}: k\in \N\}$ will be called a norm of $M$ and
will be denoted $\|M\|_\infty$. If the condition (*) is satisfied, we will
say that $M$ is {\sl $\varepsilon$-fragmented by $A$}. So,  Rosenthal's lemma
is a density lemma which asserts that ${Ros}_{M, \varepsilon}=\{A\in [\N]^\omega:
M\ \hbox{is}\  \varepsilon\hbox{-fragmented by}\ A\}$ is dense for every 
 $M\in \mathbb M$ and every $\varepsilon>0$.
Generic families
for $\mathsf{Ros}=\{{Ros}_{M, \varepsilon}: M\in \mathbb M,\ \varepsilon>0\}$ 
were introduced by D. Sobota in \cite{damian-rosenthal} and called {\sl Rosenthal families}.
It was proved in
\cite{damian-rosenthal}  that a basis of
a selective ultrafilter is a Rosenthal family and the
 following cardinal invariant which is the generic cardinal
 invariant $\mathfrak{gen}(\mathsf{Ros})$ was defined: 
$$\mathfrak{ros}=\min\{|\F|: \F \ \hbox{is a Rosenthal family}\}.$$
It was determined in \cite{damian-rosenthal} that $\mathsf{cov}(\mathcal M)\leq
\mathfrak{ros}\leq \mathfrak{u}_s$, where
$\mathfrak{u}_s$ stands for the minimal size of a base for a selective ultrafilter
if there is one, and $\mathfrak c$ otherwise.
In particular, it was proved in \cite{damian-rosenthal} that
$\mathfrak{ros}$ can be arbitrarily big on the scale of alephs and that it can be strictly smaller than
the continuum. The role of selective ultrafilters here
is natural as S. Todorcevic has shown that all of them  are $([\N]^\omega, \subseteq^*)$-generic 
over $L[\R]$, at least under a suitable large cardinal assumption (see 4.4. of \cite{farah}).
It should also be noted that the value of $\mathfrak{ros}$ in various models of set theory 
is not a  mere curiosity. As the Rosenthal lemma is a practical tool used for proving properties
of Boolean algebras, compact spaces, sequences of measures or Banach spaces, the value of
$\mathfrak{ros}$ tells us what are the sizes of the objects whose constructions
require the use of the Rosenthal lemma with the output in the constructed structure. This is related to the topic of sizes of Boolean algebras or densities of Banach spaces
with the Grothendieck property or with the Nikodym property (\cite{damian-nikodym}).
The first main result of this paper is the following:
\vskip 6pt
\begin{reptheorem}[\ref{main-theorem}]
The Rosenthal number $\mathfrak {ros}$ is equal to the reaping number
$\mathfrak r$.
\end{reptheorem}
\vskip 6pt
 So we obtain quite a full picture
of the relationships between $\mathfrak{ros}$ and other cardinal invariants 
from the Cicho\'n's and van Douwen's diagrams which is well known
for $\mathfrak r$ and can be found e.g. in \cite{blass}. In particular 
$\mathfrak{ros}$ is bounded below by $\max(\mathsf{cov}(\mathcal M),
\mathsf{cov}(\mathcal N), \mathfrak{b})$ and  above by the {\sl ultrafilter number} $\mathfrak u$
i.e., the minimal size of a base of a nonprincipal ultrafilter on $\N$. Moreover,
it is consistent that the value of $\mathfrak{ros}$ is strictly bigger or strictly smaller
than many other known cardinal invariants.

One
of the side products of the proof of 
 inequality $\mathfrak{r}\leq\mathfrak{ros}$ in
Theorem \ref{main-theorem} which is obtained in Section 4 
is the result (Corollary \ref{filter}) saying that if 
 a Rosenthal family $\F$ has cardinality less then $\mathfrak{u}$, then 
 it fails to have the strong  finite intersection property, i.e. is far from being
  a generic filter. Note that
 $\mathfrak{ros}<\mathfrak u$ is consistent by Theorem \ref{main-theorem} and the main result from
 \cite{goldstern}.
The reason why  the inequality $\mathfrak{ros}\leq\mathfrak{r}$  holds is 
that  Rosenthal's lemma has a stronger version, 
apparently overlooked by many of its users, namely:
\vskip 6pt
\begin{reptheorem}[\ref{rosenthal-stronger} (\cite{bourgain, bourgain-tzafriri})\footnote{The
 relation between the statement of Theorem \ref{rosenthal-stronger} and  Theorem 1.3$'$ of 
 \cite{bourgain-tzafriri} which is not obvious will be discussed in Section 3 after the proof
 of Theorem \ref{rosenthal-stronger}.}]
For every $\varepsilon>0$ there is $l(\varepsilon)\in \N$ such that
for every Rosenthal matrix $M$  
there is a partition $\mathcal{P} = \{ P_i : 1\leq i \leq  l(\varepsilon) \}$ 
of $\N$, such that  $M$ is  $\varepsilon\|M\|_\infty$-fragmented  by $P_i$
   for every $1\leq i \leq  l(\varepsilon) $.  
\end{reptheorem}
\vskip 6pt
This has an immediate  corollary  which answers  Question 3.18 of \cite{damian-rosenthal}:
\vskip 6pt
\begin{reptheorem}[\ref{main-ultrafilter}]
 Every nonprincipal ultrafilter over $\N$ is a Rosenthal family.
In  particular, any $\pi$-base of any nonprincipal ultrafilter is a Rosenthal family.
\end{reptheorem}
\vskip 6pt
Recall that a $\pi$-base of a nonprinipal ultrafilter $u$ is a family $\mathcal B\subseteq [\N]^\omega$
such that for every $A\in u$ there is $B\in \mathcal B$ such that $B\subseteq A$.
This result is a bit surprising at first sight because, as we mentioned before, it is consistent that there are no ultrafilters which are generic families
for  $\mathsf{Conv}$ or $\mathsf{Hom}$. 

A result like Theorem \ref{rosenthal-stronger},
where not only the density is 
asserted but actually $\N$ can be partition into sets belonging to the dense family in question will be called a {\sl paving lemma} in the analogy of to the paving conjectures equivalent to
the Kadison-Singer problem (\cite{bownik}). The remaining results presented in this paper consist
of determining generic cardinal invariants for certain collections of dense sets which  are natural subfamilies of $\mathsf{Ros}$. 
 Let us discuss these families and the results below.

Let $X, Y$ be some of the Banach spaces $c_0$ or $\ell_p$ for $1\leq p\leq\infty$. 
By $\mathcal B_0(X, Y)$
we denote the family of bounded linear operators from $X$ into $Y$ with zero diagonal,
i.e. such $T: X\rightarrow Y$  that $T(1_{\{n\}})(n)=0$ for every $n\in \N$; 
we write $\mathcal B_0(X)=\mathcal B_0(X,X)$. 
For $T\in \mathcal B_0 (X, Y)$ we define
$${Ros}_{T, \varepsilon}
=\{A\in [\N]^\omega: \|P_ATP_A\|\leq \varepsilon\|T\|\},$$
where $P_A: \ell_\infty\rightarrow \ell_\infty$ 
is given by $P_A(f)=f\cdot 1_A$ for all $f\in \ell_\infty$ and
where $1_A$ is the characteristic function of $A$. 
Let $\mathsf{Ros}(X, Y)=\{{Ros}_{T, \varepsilon}: T\in \mathcal B_0(X, Y), \varepsilon>0\}$.

It turns out that $\mathsf{Ros}=\mathsf{Ros}(c_0, \ell_\infty)$ (Proposition \ref{inclusions}
(\ref{ros=ros})) because
Rosenthal matrices correspond exactly to matrices of operators from $c_0$ to $\ell_\infty$
(Lemma \ref{operatorsc0linf})
and the $(\varepsilon\|T\|)$-fragmentation corresponds to the condition 
 in the definition of ${Ros}_{T, \varepsilon}$ (Lemma \ref{equiv-fragmenting}). In fact
the transposed matrices of Rosenthal matrices correspond exactly to
matrices of operators on $\ell_1$ (Lemmas \ref{operatorsc0linf} and \ref{operatorsl1}) and so we obtain
\vskip 6pt
\begin{reptheorem}[\ref{rosl1-main}]
$\mathfrak{ros}(\ell_1)=\mathfrak{r}$.
\end{reptheorem}
\vskip 6pt
Clearly $\mathsf{Ros}(c_0)$ is a proper subfamily of
 $\mathsf{Ros}=\mathsf{Ros}(c_0, \ell_\infty)$ (Proposition \ref{inclusions} (\ref{rosc0inros})).
A generic family for $\mathsf{Ros}(c_0)$ will be called a $c_0$-Rosenthal family
and the generic cardinal number for $\mathsf{Ros}(c_0)$ will be denoted $\mathfrak{ros}(c_0)$.
We obtain the following:
\vskip 6pt
\begin{reptheorem}[\ref{rosc0-main}]
$\mathfrak{ros}({c_0})=\min(\mathfrak{d}, \mathfrak{r}).$
\end{reptheorem}
\vskip 6pt

The paving result for operators on $\ell_2$ has been recently obtained in \cite{ks} 
which resolved the Kadison-Singer problem. This gives $\mathfrak{ros}(\ell_2)\leq \mathfrak{r}$
but we are left with:
\begin{question}
What is the value of $\mathfrak{ros}({\ell_2})$?
\end{question}
It  should be noted that the question if the paving lemma holds for
$\ell_p$ for $1<p<\infty$, $p\not=2$ is a known open problem (\cite{bownik}).
Another natural subcollection $\mathsf{Ros}_{01}$ of $\mathsf{Ros}$ consists of dense families
$${Ros}_{f}=\{A\in [\N]^\omega: f[A]\cap A=\emptyset\},$$
where $f:\N\rightarrow \N$ is a function with no fixed points. It can be easily seen that 
${Ros}_f=Ros_{M, 1/2}$,
where  $M=(m_{k, n})_{k, n\in \N}$ is a Rosenthal matrix, where
$m_{k, f(k)}=1$ and $m_{k, n}=0$ if $n\not=f(k)$ (see Section \ref{free}).
A generic family for $\mathsf{Ros}_{01}$ will be called a {\sl binary Rosenthal family}
and the generic cardinal invariant for $\mathsf{Ros}_{01}$ will be denoted by $\mathfrak{ros}_{01}$.
We obtain:
\vskip 6pt
\begin{reptheorem}[\ref{ros01-main}]
$\mathfrak{ros}_{01}=\mathfrak{r}.$
\end{reptheorem}
\vskip 6pt
 The last type of generic families we consider comes from combining 
 $c_0$-Rosenthal families and binary Rosenthal families. In fact, Rosenthal matrices
 which correspond to  elements of $\mathcal B_0(c_0)$ are exactly Rosenthal matrices whose
 columns converge to zero (Lemma \ref{operatorsc0}). 
 If we consider binary Rosenthal matrices $M$ such that
 $\mathsf{Ros}_{01}\ni {Ros}_f={Ros}_{M, 1/2}$, then one sees that the condition that the columns converge to zero translates to $f$ being finite-to-one. So we define 
 $\mathsf{Ros}_{01}(c_0)$ as the collection of families
 ${Ros}_f$ where $f:\N\rightarrow \N$ is a finite-to-one function with no fixed points.
  A generic family for $\mathsf{Ros}_{01}(c_0)$ will be called a {\sl binary $c_0$-Rosenthal family}
and the generic cardinal invariant for $\mathsf{Ros}_{01}(c_0)$ will
 be denoted by $\mathfrak{ros}_{01}(c_0)$.
We obtain:
\vskip 6pt
\begin{reptheorem}[\ref{ros01c0-main}]
$\mathfrak{ros}_{01}(c_0)=\min(\mathfrak d, \mathfrak{r}).$
\end{reptheorem}
\vskip 6pt
However we do not know the answer to the following:
\begin{question} What is the value of the generic cardinal invariant for the family
$$\mathsf{Ros}_{01}^{1-1}=\{{Ros}_f: f:\N\rightarrow \N\ 
\hbox{ is a one-to-one function with no fixed points}\}?$$
\end{question}

Let us describe the structure of of the paper.
Section 2 is devoted to proving some of the above claims concerning the relations 
between Rosenthal matrices and  linear bounded operators, functions without fixed points and
sequences of measures.
In section 3 we discuss versions of Theorem \ref{rosenthal-stronger} present in the literature and we prove it and conclude Theorem \ref{main-ultrafilter}. Section 4 is devoted to applications of nowhere reaping families which together
with the results of Section 3 give main results on the values of $\mathfrak{ros}$ and
$\mathfrak{ros}_{01}$. In section 5 we calculate $\mathfrak{ros}(c_0)$ and
$\mathfrak{ros}_{01}(c_0)$. Set-theoretic terminology is based on \cite{blass}. Terminology
concerning linear operators is introduced at the beginning of Section 2.

\section{Rosenthal matrices and families}

\subsection{Rosenthal matrices and linear bounded operators}

If $\K$ is either the field $\C$ of complex numbers or
the field $\R$ of the reals, we will consider the Banach space 
$\ell_\infty(\K) = \{ f \in \K^\N : f \mbox{ is bounded} \}$ 
with the supremum norm $||f||_\infty = \sup \{ |f(n)| : n \in \N \}$ 
and its subspace $c_0(\K) = \{ f \in \ell_\infty(\K) : \lim_{n \rightarrow \infty} f(n) = 0 \}$
as well as the spaces $\ell_p(\K)=
\{f\in \K^\N: \Sigma_{n\in \N}|f(n)|^p<\infty\}$ with the $p$-norm
$\|f\|_p= \big(\Sigma_{n\in \N}|f(n)|^p\big)^{1\over p}$, where $1\leq p<\infty$. 
We will also mention the finite-dimensional versions $\ell_\infty^n(\K)$,
$c_0^n(\K)$, $\ell_p^n(\K)$ for $n\in \N$ and $1\leq p<\infty$.
Note that $\ell_\infty^n$ is the same as $c_0^n$ for each $n\in \N$ .

We will skip
the specification of  the field $\K$, that is we will use
$\ell_\infty$, $c_0$, $\ell_p$, etc.,  as all of our arguments work for both cases.
Linear operators $T :X\rightarrow Y$
 between Banach spaces $(X, \| \ \|_{X}), (Y, \|\ \|_{Y})$ will 
  be considered with the operator norm, i.e., $\|T\|=\sup\{\|T(x)\|_Y/ \|x\|_X: x\in X\}$.
  When dealing with finite or infinite matrices we will specify 
  the norms $\|\ \|_\infty$ or $\|\ \|_1$ which are defined in the Lemmas 
  \ref{operatorsc0linf} and \ref{operatorsl1}.
Recall from the introduction that $1_A$ denotes the characteristic function of a set $A$.

We will need the following three elementary and well known lemmas on infinite matrices.
We provide the proofs for the convenience of the reader:

\begin{lemma}\label{operatorsc0linf}
Every matrix $M=(m_{k, n})_{k, n\in \N}$ satisfying
$sup\{\Sigma_{n\in \N}|m_{k, n}|: k\in \N\}=\|M\|_\infty<\infty$ defines
a linear bounded operator $T_M: c_0\rightarrow \ell_\infty$ satisfying
$$T((a_n)_{n\in \N})(k)=\Sigma_{n\in \N}a_nm_{k, n}$$ 
for each $k\in \N$.
The operator $T_M$ has norm $\|M\|_\infty$. Moreover, each bounded linear operator from
$c_0$ into  $\ell_\infty$ is of this form.
\end{lemma}
\begin{proof} The requirement concerning $M$ implies that $T_M$ is well-defined on $c_0$
into $\ell_\infty$. It is clear that $T_M$ is linear. If $(a_n)_{n\in \N}\in c_0$,
then $|T_M((a_n)_{n\in \N})(k)|\leq M_1\|(a_n)_{n\in \N}\|$, so
$T_M$ is a bounded operator and $\|T_M\|\leq\|M\|_\infty$. Having fixed
$k, i\in \N$ by taking
numbers $a_n$ such that $a_nm_{k,n}=|m_{k, n}|$ for $n\leq i$ and $a_n=0$ for $n>i$ 
we get that 
$\|T_M((a_n)_{n\in \N})\|\geq\Sigma_{n\leq i}|m_{k, n}|$ and so $\|M\|_\infty\leq \|T_M\|$.

Given any $T:c_0\rightarrow \ell_\infty$ define $m_{k, n}=T(1_{\{n\}})(k)$.
As $T^*(\delta_k)=\delta_k\circ S$ is a linear bounded functional on $c_0$ of norm
not bigger than $\|T\|$, it must be
in $\ell_1=c_0^*$. So $sup\{\Sigma_{n\in \N}|m_{k, n}|: k\in \N\}=\|M\|_\infty\leq \|T\|$.
As the span of $\{1_{\{n\}}: n\in \N\}$ is dense in $c_0$, we obtain 
$T((a_n)_{n\in \N})(k)=\Sigma_{n\in \N}a_nm_{k, n}$ for each $(a_n)_{n\in \N}\in c_0$. 
\end{proof}

\begin{lemma}\label{operatorsl1}
Every matrix $M=(m_{k, n})_{k, n\in \N}$ satisfying
$sup\{\Sigma_{k\in \N}|m_{k, n}|: n\in \N\}=\|M\|_1<\infty$ defines
a linear bounded operator $T_M: \ell_1\rightarrow \ell_1$ satisfying
$$T((a_n)_{n\in \N})(k)=\Sigma_{n\in \N}a_nm_{k, n}$$ 
for each $k\in \N$.
The operator $T_M$ has norm $\|M\|_1$. Moreover, each bounded linear
operator on $\ell_1$ is of this form.
\end{lemma}
\begin{proof} The requirement concerning $M$ implies that 
the rows of $M$ have $\ell_\infty$-norms bounded by $\|M\|_1$ as well,
so $T_M$ is well-defined on $\ell_1$ and is sending
norm one elements of $\ell_1$ into sequences bounded by $\|M\|_1$. 
It is clear that $T_M$ is linear.
$$\Sigma_{k\in \N}|T_M((a_n)_{n\in \N})(k)|\leq \Sigma_{k\in \N}\Sigma_{n\in \N}|a_n||m_{k, n}|=$$
$$=\Sigma_{n\in \N}\big(|a_n|\Sigma_{k\in \N}|m_{k, n}|\big)\leq \|M\|_1\|(a_n)_{n\in \N}\|.$$
So $T_M:\ell_1\rightarrow\ell_1$ and $\|T_M\|\leq \|M\|_1$.
 We have that
$\|T_M((1_{\{n\}})_{n\in \N})\|=\Sigma_{k\in \N}|m_{k, n}|$ and so $\|M\|_1\leq \|T_M\|$.

Given any $T:\ell_1\rightarrow \ell_1$ define $m_{k, n}=T(1_{\{n\}})(k)$.
As $T_M((1_{\{n\}})_{n\in \N})=(m_{k, n})_{n\in \N}$, it follows that
$sup\{\Sigma_{k\in \N}|m_{k, n}|: n\in \N\}\leq \|T_M\|$. As the span of $\{1_{\{n\}}: n\in \N\}$
 is dense in $\ell_1$, we obtain 
$T((a_n)_{n\in \N})(k)=\Sigma_{n\in \N}a_nm_{k, n}$ for each $(a_n)_{n\in \N}\in \ell_1$. 
\end{proof}

\begin{lemma}\label{operatorsc0}
Every matrix $M=(m_{k, n})_{k, n\in \N}$ satisfying
$sup\{\Sigma_{n\in \N}|m_{k, n}|: k\in \N\}=\|M\|_\infty<\infty$ 
and $\lim_{k\rightarrow \infty}m_{k, n}=0$ for each $n\in \N$ defines
a linear bounded operator $T_M: c_0\rightarrow c_0$ satisfying
$$T((a_n)_{n\in \N})(k)=\Sigma_{n\in \N}a_nm_{k, n}$$ 
for each $k\in \N$.
The operator $T_M$ has norm $\|M\|_\infty$. Moreover, each bounded linear operator on
$c_0$  is of this form.
\end{lemma}
 \begin{proof}
 By Lemma \ref{operatorsc0linf} the operator $T_M:c_0\rightarrow\ell_\infty$
 has norm equal to $\|M\|_\infty$. $T_M(1_{\{n\}})=(m_{k, n})_{k\in \N}\in c_0$
 so $\lim_{k\rightarrow \infty}m_{k, n}=0$. For the moreover part use again
 Lemma \ref{operatorsc0linf} to conclude that any operator on $c_0$
 is given as in Lemma \ref{operatorsc0linf}. The same argument
 concerning $T_M(1_{\{n\}})$ as above yields $\lim_{k\rightarrow \infty}m_{k, n}=0$.
 \end{proof}
 
 Before we note the relations between various generic families we need one more observation:
 
 \begin{lemma}\label{equiv-fragmenting}
  Suppose that $M$ is a Rosenthal matrix with zero diagonal and $A\subseteq \N$
 is infinite.
 $M$ is $\varepsilon$-fragmented by $A$ if and only if $\|M_AMM_A\|_\infty\leq \varepsilon$,
 where $M_A=(m_{k, n})_{k,n\in \N}$ is the diagonal matrix  satisfying
 $m_{k, k}=1$ if $k\in A$ and $m_{k, k}=0$ otherwise.
 \end{lemma}
 \begin{proof}
 $M$ is $\varepsilon$-fragmented by
$A$ if and only if   $\Sigma_{k\in A}m_{k, n}\leq\varepsilon$ for each $k\in A$
if and only if 
$\|M_AMM_A\|_\infty\leq\varepsilon$. 
 \end{proof}

\begin{proposition}\label{inclusions} $ $
\begin{enumerate}
\item $\mathsf{Ros}=\mathsf{Ros}(c_0, \ell_\infty)$,\label{ros=ros}
\item $\mathsf{Ros}(c_0, \ell_\infty)=\mathsf{Ros}(\ell_1)$,\label{ros=rosl1}
\item $\mathsf{Ros}(c_0)\subseteq\mathsf{Ros}$.\label{rosc0inros}
\end{enumerate}
\end{proposition}
\begin{proof}
Let $M$ be a Rosenthal matrix and $\varepsilon>0$. Let $M'=(m_{k, n}')_{k, n\in \N}$
 be obtained from $M$
by replacing the diagonal entries by zeros. We will see that 
$${Ros}_{M, \varepsilon}={Ros}_{T_{M'}, \varepsilon/\|T_{M'}\|}=\{A\in [\N]^\omega: 
\|P_AT_{M'}P_A\|\leq \varepsilon\},$$
where $T_{M'}$ is the operator defined in Lemma \ref{operatorsc0linf} for $M'$. 
Let $A\subseteq \N$ be infinite.
$M$ is $\varepsilon$-fragmented by
$A$ if and only if  $M'$ is $\varepsilon$-fragmented by
$A$ if and only if
$\|M_AM'M_A\|_\infty\leq\varepsilon$ (by Lemma \ref{equiv-fragmenting})
if and only if 
$\|P_AT_{M'}P_A\|\leq\varepsilon$ by Lemma \ref{operatorsc0linf}. 

Now let $T:c_0\rightarrow \ell_\infty$ has zeros on the diagonal, 
i.e., $T(1_{\{n\}})(n)=0$  for each $n\in \N$. Define a matrix
$M=(m_{k, n})_{k, n\in \N}$ given by $m_{k, n}=|S(1_{\{n\}})(k)|$ for all
$k, n\in \N$. Lemma \ref{operatorsc0linf} implies that it is a Rosenthal matrix. 
Let $A\subseteq \N$ be infinite.
$\|P_ATP_A\|\leq \varepsilon\|T\|$ if and only if  $\|M_AMM_A\|\leq \varepsilon\|T\|$
if and only if 
$M$ is $\varepsilon\|T\|$-fragmented by
$A$ by Lemma \ref{equiv-fragmenting}. So
${Ros}_{M, \varepsilon\|T\|}={Ros}_{T, \varepsilon}$ which proves (1).

For (2) we note that by Lemmas \ref{operatorsc0linf}
and \ref{operatorsl1} a matrix $M$ defines an operator $T_M$ from $c_0$ into
$\ell_\infty$ if and only if its transpose $M'$  defines an operator $T_{M'}$  on $\ell_1$.
Moreover $\|M\|_\infty=\|M'\|_1$. So
$\|P_A T_M P_A\|=\|P_A M P_A\|_\infty=\|P_A {M'} P_A\|_1=\|P_A T_{M'} P_A\|$ and
consequently (2) is proved.

(3) follows from the fact that
operators on $c_0$ form a subclass of operators from $c_0$ into $\ell_\infty$.
\end{proof}

So we immediately obtain:

\begin{proposition}\label{inequalities-operators} 
$  $
\begin{enumerate}
\item $\mathfrak{ros}=\mathfrak{ros}(\ell_1)$,\label{ros-rosl1}
\item $\mathfrak{ros}(c_0)\leq\mathfrak{ros}.$\label{rosc0-ros}
\end{enumerate}
\end{proposition}

\subsection{Rosenthal families and  free sets}\label{free}

From a combinatorial point of view, a natural special kind of matrices $(m_{k, n})_{k, n\in \N}$ as
in Definition \ref{rosenthal-intro} is defined by requiring that it is binary 
(i.e., $m_{k, n}\in \{0,1\}$ for all $k, n\in \N$), 
 antidiagonal (i.e., $m_{k, k}=0$ for all $k\in \N$)  and each row has a nonzero entry, i.e., 
 there is a function $f:\N\rightarrow \N$ with no fixed points
 such that $m_{k, f(k)}=1$ for each $k\in \N$
 and $m_{k, n}=0$ for each $n\not=f(k)$. We will denote such a matrix by $M_f$.

 \begin{lemma}\label{equiv-free} Suppose 
 that $f:\N\rightarrow\N$ has no fixed points and $A\subseteq \N$
 is infinite.
 $M_f$ is  $1/2$-fragmented by $A$ if and only if
   $f[A]\cap A=\emptyset$. 
  \end{lemma}
  \begin{proof}
  $M_f$ is $1/2$-fragmented by
$A$ if and only if   $\Sigma_{k\in A}m_{k, n}=\leq1/2$ for each $k\in A$
if and only if $f(k)\not\in A$ for each $k\in A$ if and only if 
$f[A]\cap A=\emptyset$. 
  \end{proof}
  
  A set $A$ satisfying $f[A]\cap A=\emptyset$ is called {\sl free} for $f$ following
  a well-established combinatorial terminology (e.g. \cite{komjath}) according to which,
   more generally 
  given a set mapping $f : X \rightarrow \wp(X)$ a set $Y\subseteq X$ is called {free } if  
  $y\not\in f(y')$  for any $y, y'\in Y$.

  \begin{proposition}\label{inequalities}$  $
  \begin{enumerate}
  \item $\mathfrak{ros}_{01}\leq \mathfrak{ros}$.\label{ros01-ros}
  \item $\mathfrak{ros}_{{01}}({c_0})\leq \mathfrak{ros}_{{01}}$.\label{ros01c0-ros01}
  \item $\mathfrak{ros}_{{01}}({c_0})\leq \mathfrak{ros}(c_0)$.\label{ros01c0-rosc0}
  \end{enumerate}
  \end{proposition}
\begin{proof} For (\ref{ros01-ros}) we note that by Lemma \ref{equiv-free} we have
${Ros}_f={Ros}_{M_f, 1/2}$, where $f:\N\rightarrow \N$ has no fixed points. So
we have $\mathsf{Ros}_{01}\subseteq \mathsf{Ros}$ and this implies (\ref{ros01-ros}).

For (\ref{ros01c0-ros01}) we need to note that
$\mathsf{Ros}_{01}(c_0)\subseteq \mathsf{Ros}_{01}$ which
follows from the inclusion of finite to one functions with no fixed points in
all functions with no fixed points.

For (\ref{ros01c0-rosc0}) we need to note that
$\mathsf{Ros}_{01}(c_0)\subseteq \mathsf{Ros}(c_0)$ which follows from
the fact that if $f:\N\rightarrow \N$ has no fixed points and is finite
to one, then $M_f$ is a Rosenthal matrix whose columns have only finitely
many non-zero entries and so by Lemma \ref{operatorsc0} the matrix $M_f$
corresponds to a linear bounded operator $T_{M_f}$  on $c_0$. Moreover,
by Lemmas \ref{equiv-fragmenting} and \ref{equiv-free} the conditions
$\|P_AT_{M_f}P_A\|\leq 1/2$ corresponds to $f[A]\cap A=\emptyset$ 
for any infinite $A\subseteq \N$.
\end{proof}

\subsection{Rosenthal families and sequences of measures}\label{measures}
In this section we show that the generic cardinal invariant $\mathfrak{ros}$ 
corresponding to the combinatorial version of Rosenthal's lemma (Lemma \ref{rosenthal-intro})
is the same for the families of dense sets corresponding to both of the classical versions of Rosenthal's lemma (Lemma \ref{rosenthal-lemma}).
H. Rosenthal proved in \cite{rosenthal1} and
 \cite{rosenthal2} the following\footnote{Originally
  the lemma served for proving fundamental new results  
  in the theory of Banach spaces in \cite{rosenthal1} and \cite{rosenthal2} but it has been 
  also quickly realized that it may serve to simplify
   an extensive body of results from the geometry of Banach spaces due to Phillips, Grothendieck, 
  Pe\l czy{\'n}ski, Lindenstrauss and others
(\cite{phillips, grothendieck, pelczynski1,  lindenstrauss}).
It turned out that the lemma may play even more 
dramatic role in the the theory of Banach valued
 vector measures, essentially simplifying the proofs of Diestel-Faires and Orlicz-Pettis
  theorems, the Nikodym boundedness principle and many others surveyed in I. 4 of \cite{diestel-uhl}.}: 

\begin{lemma}\label{rosenthal-lemma} Let $\mathcal A$
 be a Boolean algebra and $\mu_k: \A\rightarrow \R_+\cup\{0\}$
be  finitely additive measures on $\A$ for each $k\in \N$ which are uniformly bounded i.e.,
$\mu_k(1_\A)\leq \rho$ for some $\rho\geq 0$, 
where $1_\A$ is the unit of $\A$. Let $(A_n)_{n\in \N}$ be pairwise disjoint
elements of $\A$ and $\varepsilon>0$. Then there is an infinite $A\subseteq \N$ such that
for every $k\in A$ we have 
$$\sum_{n\in A\setminus\{k\}}\mu_k(A_n)\leq\varepsilon.$$
Moreover, if $\A$ is a $\sigma$-complete Boolean algebra (but still the measures are
assumed only to be finitely additive), then
$A$ above may be chosen to satisfy the following stronger requirement for each $k\in A$:
$$\mu_k\Big(\bigvee_{n\in A\setminus\{k\}}A_n\Big)\leq\varepsilon,$$
where $\bigvee$ denotes the supremum in $\A$.
\end{lemma}

Given an infinite Boolean algebra $\A$ by $ac(\A)$ we will denote  the class of all
 infinite pairwise disjoint  sequences ${\overline A}=(A_n)_{n\in \N}$ of
elements of $\A$. By $\mu_\infty(\A)$ we will denote the class of  all uniformly bounded sequences
of finitely additive measures on $\A$. Given $\varepsilon>0$, a Boolean algebra
$\A$, ${\overline A}\in ac(\A)$ and ${\overline \mu}\in \mu_\infty(\A)$ we can consider
$${Ros}_{\A, {\overline A}, {\overline \mu}, \varepsilon}=
\{A\in [\N]^\omega:\forall k\in A\  \sum_{n\in A\setminus\{k\}}\mu_k(A_n)\leq\varepsilon\}.$$

\begin{proposition}\label{measures-rosenthal}
The generic families for
the collection of dense sets of the form
${Ros}_{\A, {\overline A}, {\overline \mu}, \varepsilon}$,
where $\A$ is a Boolean algebra, ${\overline A}\in ac(\A)$,
${\overline \mu}\in \mu_\infty(\A)$ and $\varepsilon>0$ are
exactly Rosenthal families. Consequently the generic cardinal invariant
for these collections is $\mathfrak{ros}$.
\end{proposition}
\begin{proof}
It is clear that ${Ros}_{\A, {\overline A}, {\overline \mu}, \varepsilon}={Ros}_{M, \varepsilon}$,
where $M=(m_{k, n})_{k, n\in \N}$ is a Rosenthal matrix defined by
$m_{k, n}=\mu_k(A_n)$ for every $k, n\in \N$. So the corresponding notions of generic families
and the generic cardinal invariants are the same for the above collection of
dense sets and for $\mathsf{Ros}$. 

On the other hand if $M=(m_{k, n})_{k, n\in \N}$ is a 
Rosenthal matrix, then we can define
a finitely additive measure $\mu_k(A)=\Sigma_{n\in A}m_{k, n}$
for any finite or cofinite $A\subseteq\N$. Now
${Ros}_{\A, {\overline A}, {\overline \mu}, \varepsilon}={Ros}_{M, \varepsilon}$,
where $\A$ is the algebra of finite  or cofinite subsets of $\N$,
${\overline A}=(\{n\})_{n\in \N}$.
\end{proof}

 Dense sets corresponding to the
second version of Rosenthal's lemma above are of the form
$${Ros}_{\A, {\overline A}, {\overline \mu}, \varepsilon}^\sigma=
\{A\in [\N]^\omega:\forall k\in A\  
\mu_k\Big(\bigvee_{n\in A\setminus\{k\}}A_n\Big)\leq\varepsilon\},$$
where $\A$ is a $\sigma$-complete Boolean algebra.  
To show that the generic cardinal invariant corresponding to the dense sets from the second 
version of Rosenthal's lemma (Proposition \ref{sigma-rosenthal}) we need the following:

\begin{proposition}\label{avoiding-boundary} Suppose that
$(C_\xi)_{\xi<\omega_1}$ is an almost disjoint family of infinite subsets of
$\N$. 
If $\A$ is a  Boolean algebra and $\mu_k$s  for $k\in \N$ are  finitely additive
positive measures on $\A$ whose norms are  bounded by $\rho\in \R$, and if $(A_n)_{n\in \N}$ are pairwise disjoint
elements of $\A$, then for all but countably many $\xi<\omega_1$ for every
$k\in \N$
$$\mu_k\Big(\bigvee_{n\in B}A_n\Big)=\sum_{n\in B}\mu_k(A_n) $$
whenever $B\subseteq C_\xi$ and $\bigvee_{n\in B}A_n$ exists in $\A$.
\end{proposition}
\begin{proof}
This is basically the argument from \cite{rosenthal1}.
Consider the Stone
 space $K_\A$ of the Boolean algebra $\A$ and its clopen sets $[A]$ which are those
 ultrafilters of $\A$ which contain $A$. The measures $\mu_k$s define linear functionals
 of norm not bigger than $\rho$ on the subspace of $C(K)$ consisting of continuous functions with finitely many values. By the Hahn-Banach theorem they extend to the entire $C(K)$ preserving the norm,
 and by the Riesz representation theorem the extensions can be associated with countably additive
 Borel regular measures on $K$. We will denote these extensions by the same letters $\mu_k$.
 
 Suppose that the lemma fails, so there is an uncountable set $X\subseteq\omega_1$, $k\in \N$ and
 infinite $B_\xi\subseteq C_\xi$ such that $\bigvee_{n\in B_\xi}A_n$ exists in $\A$
 but $\mu_k\Big(\Delta_{B_\xi}\Big)>0$ for each $\xi\in X$, where
 $$\Delta_B=[\bigvee_{n\in B}A_n]\setminus \bigcup_{n\in B} [A_n]$$
 for any $B\subseteq \N$ for which the supremum $\bigvee_{n\in B}A_n$ exists in $\A$.
 Now  one notes that $[\bigvee_{n\in B}A_n]$ is a disjoint union of 
 $[\bigvee_{n\in B\setminus F}A_n]$ and $\bigcup_{n\in F}[A_n]$ for any finite $F\subseteq B$.
  With this we conclude that $\Delta_B=\Delta_{B'}$
 if $B$ and $B'$ differ by a finite set, and so 
 $\Delta{B_\xi}\cap \Delta{B_{\xi'}}=\emptyset$ if $\xi\not=\xi'$.
 But a bounded Radon measure cannot be nonzero on uncountably many pairwise disjoint sets,
 so there is $\xi<\omega_1$ such that $\mu_k(\Delta_\xi)=0$ for each $k$. This
 contradicts  the choice of $B_\xi$.
\end{proof}

\begin{proposition}\label{sigma-rosenthal}
The generic cardinal invariant for the collection of all dense sets
of the form ${Ros}_{\A, {\overline A}, {\overline \mu}, \varepsilon}^\sigma$, where
$\A$ is a $\sigma$-complete Boolean algebra, ${\overline A}\in ac(\A)$,
${\overline\mu}\in \mu_\infty(\A)$, $\varepsilon>0$ is equal to $\mathfrak{ros}$.
\end{proposition}
\begin{proof}
Let $\F$ be a generic family for the collection of the dense sets as in
the proposition. We will show that it is a Rosenthal family. 
Let $M=(m_{k, n})_{k, n}$ be a Rosenthal matrix and $\varepsilon>0$.
Define measures $\mu_k(A)=\Sigma_{n\in A}m_{k, n}$ for $A\subseteq \wp(\N)$.
As in the $\sigma$-complete Boolean algebra $\wp(\N)$  the suprema are infinite unions
and the above measures are $\sigma$-additive we have
$$Ros_{M, \varepsilon}={Ros}_{\wp(\N), {\overline A}, {\overline \mu}, \varepsilon}^\sigma,$$
where ${\overline A}=(\{n\})_{n\in \N}$ and ${\overline \mu}=(\mu_k)_{k\in \N}$ and hence
$\F$ is a Rosenthal family.

Now suppose $\F$ is a Rosenthal family. Consider $\F'\subseteq [\N]^\omega$
such that below each element $A\in \F$ there is in $\F'$ an almost disjoint family of size $\omega_1$ of
infinite subsets of $A$. As $\F$ is uncountable, it is easy to
construct such $\F'$ of the same uncountable size as $\F$. We will
shot that $\F'$ meets each dense set as in the proposition.

Let $\A$ be a $\A$ is a $\sigma$-complete Boolean algebra, ${\overline A}\in ac(\A)$,
${\overline\mu}\in \mu_\infty(\A)$, $\varepsilon>0$. 
Using the fact that $\F$ is a Rosenthal family find 
$A\in \F\cap  {Ros}_{\A, {\overline A}, {\overline \mu}, \varepsilon}$. 
By Lemma \ref{avoiding-boundary} there is $A'\in \F'$ as in this lemma which implies
that $A'\in \F'\cap {Ros}_{\A, {\overline A}, {\overline \mu}, \varepsilon}^\sigma$.

\end{proof}

\section{Paving lemma for sets fragmenting Rosenthal matrices}

The main purpose of this section is to prove Theorem \ref{rosenthal-stronger} and
\ref{main-ultrafilter}. This is based on a paving lemma which can be concluded from
a paving lemma due to Bourgain (\cite{bourgain},\cite{bourgain-tzafriri}, see comments below
Theorem \ref{rosenthal-stronger}) But we provide our original proof because it is purely combinatorial.
We will need the following two lemmas that deal
  with triangular matrices: 

\begin{lemma}\label{ult1}
Let $M = (m_{k,n})_{k, n \in \N}$ be a Rosenthal matrix  such
 that $\|M\|_\infty\leq 1$ and 
 for every $n \in \N$ and every $n\geq k$ we have $m_{k, n} = 0$, then
  for every positive $l\in \N$ there is a partition 
  $\mathcal{P}_0 = \{ P_i : 1\leq i \leq  l  \}$ of $\N$
   such that  $M$ is  $({1\over l})$-fragmented  by $P_i$
   for every $1\leq i \leq  l $. 
\end{lemma}

\begin{proof} Fix $M\in \M$ and $l\in \N$ as in the lemma.
We will recursively  construct a function $f: \N \rightarrow \{1, ..., l\}$ 
such that for every $1\leq i \leq  l $ the set $P_i = f^{-1}[{i}]$, 
$\frac{1}{l}$-fragments $M$:
 If $j \leq l $, then let $f(j) = j$.
  Suppose that $f(j)$ has been constructed for every $j < j_0$, 
  we will construct $f(j_0)$. 
  For $1\leq i \leq  l $ let $P_i^{j_0} = f^{-1}[{i}]$,
  Observe that $\{P_i^{j_0} : 1\leq i \leq  l  \}$ is a 
  partition of $[1,j_0)$. Using the fact that $M$ has its norm not bigger
  than $1$ and the pigeonhole principle 
  it is possible to pick  $1\leq i \leq  l $
  such that $\displaystyle \sum_{n\in P_i^{j_0}}m_{j_0,n} \leq \frac{1}{l}$.
  So we put  $f(j_0) = i$.
   This finishes the construction. 
   It follows that the partition $\{f^{-1}(i): 1\leq i \leq  l \}$
    is the partition we are looking for.
\end{proof}

\begin{lemma}\label{ult2}
Let $M = (m_{k,n})_{k, n \in \N}$ be a Rosenthal matrix  such
 that $\|M\|_\infty\leq 1$ and 
 for every $n \in \N$ and every $n\leq k$ we have $m_{k, n} = 0$, then for every positive
  $l\in \N$ there is a 
  partition $\mathcal{P}_1 = \{ P_i: 1\leq i\leq l \}$ of $\N$ 
   such that  $M$ is  $({1\over l})$-fragmented  by $P_i$
   for every $1\leq i \leq  l $. 
\end{lemma}
\begin{proof}

 Fix $M\in \M$ and $l\in \N$ as in the lemma. Recursively with respect to $|F|$ we will find
  for each $F \in [\N]^{<\omega}$, 
  a function $f_F:F \rightarrow \{1, ..., l\}$ 
  such that  for every $1\leq i \leq  l $ the set $P_i = f^{-1}[{i}]$, 
$\frac{1}{l}$-fragments $M$:
  Clearly this can be done if $|F| = 1$, 
  so suppose that we already constructed $f_F$ for every
   $F$ such that $|F| < j$ and let $G \in [\N]^{j}$. 
   Let $g = \min G$ and let $F = G \setminus \{g\}$. 
   For $1\leq i \leq  l$, let $A_i = f_F^{-1}[\{i\}]$. 
   Next, we use the fact that $M$ has its norm not bigger
  than $1$
   and the pigeonhole principle to find $1\leq i \leq  l$ 
   such that $\sum_{n \in A_i}m_{g,n} \leq \frac{1}{l} $ 
   and let $f_G = \langle g,i \rangle \cup f_F$. It follows that $f_G$ has the desired properties.

To finish the proof, observe that $\{1, ..., l\}^\N$ is a compact metrizable space, 
so $(f_n)_{n\in \N}$, where $f_n=f_{\{0, ..., n\}}$ 
has a convergent subsequence $(f_{n_k})_{k\in \N}$.
It follows that any finite fragment $f|\{0, ..., n\}$ of $f$ agrees with 
some $f_{n_k}$ for some $k\in \N$ on $\{0, ..., n\}$, this means that
$\{ f^{-1}[\{i\}] : 1\leq i \leq  l\}$ is the partition that we are looking for. 
\end{proof}

\begin{theorem}[\cite{bourgain, bourgain-tzafriri}]\label{rosenthal-stronger}
For every $\varepsilon>0$ there is $l(\varepsilon)\in \N$ such that
for every Rosenthal matrix $M$ an 
there is a partition $\mathcal{P} = \{ P_i : 1\leq i \leq  l(\varepsilon) \}$ 
of $\N$, such that  $M$ is  $\varepsilon\|M\|_\infty$-fragmented  by $P_i$
   for every $1\leq i \leq  l(\varepsilon) $.  
\end{theorem}
\begin{proof} Fix $\varepsilon>0$ and positive  $\l\in \N$ such that $2/l<\varepsilon$.
Let $M = (m_{k, n})_{k, n \in \mathbb{N}}$. We may assume that it is nonzero.
Let $M_0 = (m^0_{k, n})_{k, n \in \mathbb{N}}$ 
be the matrix defined by 
$m^0_{k, n} = m_{k, n}/\|M\|_\infty $ if and only if $k < n$, otherwise $m^0_{k, n} =0$ 
and let $M_1 = (m^1_{k, n})_{k, n \in \mathbb{N}}$ be
 the matrix defined by $m^1_{k, n} = m_{k, n}/\|M\|_\infty $ if and only if $k > n$, 
 otherwise $m^0_{k, n} =0$. 
 Apply Lemma \ref{ult1} and Lemma \ref{ult2} for $M_0$ and $M_1$ respectively 
  to obtain two partitions of $\N$, namely $\mathcal{P}_0 = \{ P_i^0 : 1\leq i \leq  l \}$ and 
 $\mathcal{P}_1 = \{ P_i^1 : 1\leq i \leq  l \}$ 
 with the properties stated in those lemmas. 
 Consider the family $\mathcal{P} = 
 \{ P_i^0 \cap P^1_j : 1\leq i \leq  l^2\}$.
 It is clear that $M=\|M\|_\infty M_0+\|M\|_\infty M_1$ is $({2\over l})\|M\|_\infty$-fragmented
 by any $A\in \mathcal P$.
\end{proof}

The sources \cite{bourgain} and \cite{bourgain-tzafriri} contain paving lemmas for
operators on $\ell_1^n$. In fact, the number of pieces of
the partition there is better than ours. It is also called a matrix-splitting lemma 
 in Section 4.1 of \cite{handbook}. It is was well known that using a compactness type argument like
 in the proof of Lemma \ref{ult2} one can obtain from these versions a paving lemma for 
 infinite dimensional $\ell_1$. From this using Lemmas \ref{operatorsc0linf} and
 \ref{operatorsl1} one can obtain the above paving lemma for Rosenthal matrices.
 After we proved Theorem \ref{rosenthal-stronger} and realized that  it yields a paving lemma 
 for  operators on $c_0$ we asked 
   B. Johnson and G. Schechtman if it was already known. We are grateful to them for
  providing the above  information  and indicating the references 
  \cite{bourgain, bourgain-tzafriri, handbook}. It should be added that
  a paving lemma for binary matrices like in Section \ref{free} 
has already been noted by P. Erd\"os in 1950 (p. 137 of \cite{erdos}, see Ex. 26.9 of \cite{komjath}
for a proof). We are grateful to P. Komjath for providing us with this reference. In fact
the compactness arguments used to pass from the finite to the infinite matrices could be seen as 
a version of an application of  de Bruijn-Erd\"os theorem which says that the chromatic number of an infinite graph is  $\leq k$ if and only if the chromatic number of   every of its
finite subgraph is $\leq k$, where $k\in \N$ (\cite{dberdos}).

In \cite{damian-rosenthal}, D. Sobota proved that every selective ultrafilter is a
 Rosenthal family and asked whether this is the case for ultrafilters 
 in general  (Question 3.18). The following is 
a positive answer to this question.

\begin{theorem}\label{main-ultrafilter} Every nonprincipal ultrafilter over $\N$ is a Rosenthal family.
In  particular, any $\pi$-base of any nonprincipal ultrafilter is a Rosenthal family.
\end{theorem}
\begin{proof} 
Let $u$ be a nonprincipal ultrafilter on $\N$.
Fix a matrix
 $M=(m_{k, n})_{k, n\in \N}$ and $\varepsilon>0$.   
 Apply Theorem \ref{rosenthal-stronger} for $M$ and
 $\varepsilon/\|M\|_\infty$ obtaining a partition of $\N$ consisting
 of sets which $\varepsilon$-fragment $M$. One  element
of the partition must be a member of the ultrafilter $u$. 
\end{proof}

We note that  a paving lemma
is not true for an arbitrary bounded linear $T: \ell_\infty\rightarrow \ell_\infty$
 satisfying $T(1_{\{n\}})(n)=0$ for each $n\in\N$:  Let $u$  be a
 nonprincipal ultrafilter over $\N$.
 Define $T(f)(k)=\lim_{n\in u}f(n)$ for each $k\in \N$, i.e, 
 the range of $T$ are constant sequences. It is clear that $\|T\|=1$. Since $u$
 is nonprincipal it follows that $T(1_{\{n\}})=0$ for each $n\in \N$.
 Given any partition $\{A_1, ..., A_l\}$ of $\N$ for some $l\in \N$ 
  there is $1\leq i\leq l$ such that $A_i\in u$, so
  $P_{A_i}TP_{A_i}(\chi_{ A_i})=P_{A_i}(\chi_\N)=\chi_{A_i}$, so
  $P_{A_i}TP_{A_i}$ has norm one.

\section{The Rosenthal number and the reaping number}

\begin{definition} A  family $\mathcal F$ of infinite subsets
of $\N$ is called {\sl nowhere reaping} if for every $B\subseteq [\N]^\omega$ satisfying
$\{A\cap B: A\in \F\}\subseteq [\N]^\omega$, there is 
$C_B\subseteq B$ such that $A\cap B\cap C_B$ and $A\cap B\setminus C_B$ are both
infinite for all $A\in \F$.
\end{definition}

Note that a subfamily of  $[\N]^\omega$  of size smaller than $\mathfrak{r}$ are nowhere reaping.

\begin{lemma}\label{hereditary} If $\mathcal F\subseteq [\N]^\omega$ is nowhere reaping,
then it is not a binary Rosenthal family. 
\end{lemma}
\begin{proof}
Let  $\mathcal F\subseteq [\N]^\omega$ be nowhere reaping.
 We will construct $f:\N\rightarrow \N$ with no fixed points such that
 $f[A]=\N$ for every $A\in\F$. This will show that 
 $\F$ is not a binary Rosenthal family.

 First, by recursion in $n\in \N$, we construct
 a pairwise disjoint family $\{ B_n : n \in \N\}$ of infinite subsets of $\N$ such that
 $B_n\cap A$ and
  $A \setminus (\bigcup_{i \leq n}B_i)$ are infinite
     for each $A\in\F$ and  for every $n\in \N$. The existence of $B_0$ 
     follows from the fact that $\F$ is not reaping. The
inductive hypothesis and the fact that $\F$ is nowhere reaping 
applied below $\N\setminus\bigcup_{i\leq n}B_i$ produces the next set $B_{n+1}$ 
in the inductive step of
the construction of $\{ B_n : n \in \N\}$.

 This induces an entire function $f:\N\rightarrow\N$ defined by
 $f(k)=n$ if $n\not=k\in B_n$ and $f(k)=k+1$ if $k$ is not in any $B_n$ or if $k\in B_k$. Clearly
 $f$ has no fixed points and 
 $f^{-1}[\{n\}]\supseteq (B_n\setminus\{n\})\cap A\not=\emptyset$ 
  for every $A\in \F$. It follows that $f[A]=\N$ for every $A\in \F$
 as required.

\end{proof}

\begin{corollary}\label{filter-ultra} If $\F$ is a filter and a Rosenthal family, then there
is  an infinite $A\subseteq \N$ such that
$\{F\cap A: F\in \F\}\cup\{\N\setminus A\}$ generates an ultrafilter.
\end{corollary}
\begin{proof} By Lemma \ref{hereditary} there is infinite $B\subseteq\N$
such that $\{F\cap  B: F\in \F\}\subseteq [\N]^\omega$ and
there is no infinite $C_B\subseteq B$ which splits $\{F\cap  B: F\in \F\}$.
So for any $C\subseteq B$ either $C\in \F$ or $B\setminus C\in \F$, so
the corollary follows. 

\end{proof}

Recall that a family $\F\subseteq[\N]^\omega$ has {\sl the strong finite intersection
property} if the intersection of every finite subfamily of $\F$ is infinite.

\begin{corollary}\label{filter}
No (binary) Rosenthal family of cardinality smaller than $\mathfrak{u}$ 
has the strong finite intersection property.
\end{corollary}
\begin{proof} Suppose that $\F$ 
of cardinality smaller than $\mathfrak{u}$ has the strong finite intersection property.
Let $\F'$ be the filter generated by $\F$. By Corollary \ref{filter-ultra}
we would obtain a nonprincipal ultrafilter generated by less than $\mathfrak{u}$ elements.
\end{proof}

\begin{proposition} $\mathfrak{r} \leq \mathfrak{ros}_{01}$. \label{r-ros01}
\end{proposition}
\begin{proof}
Suppose that $\F$ has cardinality smaller then $\mathfrak{r}$.
Then it is nowhere reaping and so by 
Lemma \ref{hereditary} the family
$\F$ is not a binary Rosenthal family. 
\end{proof}

\begin{proposition}\label{ros-r}
$\mathfrak{ros}  \leq \mathfrak{r}$.
\end{proposition}
\begin{proof}
 Let $\F\subseteq [\N]^\omega$, of size $\mathfrak{r}$,  be such a reaping family 
 that for every $A\in \F$, 
the set $\F\cap \mathcal{P}(A)$ 
is a reaping family of size $\mathfrak r$ and $\F$ is closed under finite modifications. Such family
 can be easily constructed (see 3.7 of \cite{reaping} where such families are called hereditarily reaping).

Let $M \in \mathbb{M}$ and $\varepsilon>0$. 
We will see that 
 there is an $F \in \F$ such that $M$ is $\varepsilon$-fragmented by $F$.
 Let $\{A_1, ..., A_{l}\}$ be a partition of $\N$ 
 such that each piece $\varepsilon$-fragments $M$ 
 which exists by Theorem \ref{rosenthal-stronger}.
 Using the fact
 that $\mathcal F$ is 
 reaping and closed under finite modifications find $F_1 \in \F$ such that either $F_1$
  is disjoint with $A_1$ or $F_1$ is contained in $A_1$. 
  If $F_1 \subseteq A_1$ then $F=F_1$ works. 
  Otherwise it is possible to pick an infinite
  $F_2 \subseteq F_1$ in $\F$ such that either 
  $F_2 \subseteq A_2$ or $F_2$ is disjoint from $A_1 \cup A_2$.
   If we follow this process, 
   it is evident that we will eventually 
   find $F_i \in \F$ for 
   $1\leq i \leq l$ such that $F_i \subseteq A_i$. Then 
    $F=F_i$ is as required.
    \end{proof}

An alternative proof of Proposition \ref{ros-r} is to use
a theorem of Balcar and Simon from \cite{balcar-simon} which says that
the reaping number $\mathfrak r$ is the minimal 
size of a $\pi$-base of a nonprincipal ultrafilter on $\N$. 
Such a $\pi$-base is a Rosenthal family by the second part of  Theorem \ref{main-ultrafilter}.

\begin{theorem}\label{main-theorem}
The Rosenthal number $\mathfrak {ros}$ is equal to the reaping number
$\mathfrak r$.
\end{theorem}
\begin{proof}
By Propositions \ref{inequalities} (\ref{ros01-ros}),
\ref{ros-r} and \ref{r-ros01} we have $\mathfrak{r}\leq \mathfrak{ros}_{01}\leq
\mathfrak{ros}\leq\mathfrak{r}$.
\end{proof}

The argument in the above proof also gives:

\begin{theorem}\label{ros01-main}$  $
$\mathfrak {ros}_{01}=\mathfrak{r}$.
\end{theorem}

\begin{theorem}\label{rosl1-main}
$\mathfrak{ros}(\ell_1)=\mathfrak{r}$.
\end{theorem}
\begin{proof} Use Theorem \ref{main-theorem} and Proposition \ref{inequalities-operators}
(\ref{ros-rosl1}).
\end{proof}

\section{$c_0$-Rosenthal numbers}

In this section we calculate the $c_0$-Rosenthal number $\mathfrak{ros}(c_0)$ and the
binary $c_0$-Rosenthal number $\mathfrak{ros}_{01}(c_0)$ - see the introduction for the definitions.
First let us recall some terminology.
If $f,g \in \N^\N$, then $g$ eventually dominates $f$, 
  denoted by $f \leq^* g$, if there is a $n \in \N$ such that 
  for every $k > n$, $f(k) \geq g(k)$. 
  $D \subseteq \N^\N$ is a dominating family
   if every $f \in \N^\N$ is dominated by some member of $D$. 
   The dominating number $\mathfrak{d}$
    is the smallest size of a dominating family. 
  Following Definition 2.9 of \cite{blass} an interval partition is a partition of $\N$ into (infinitely
many) finite intervals $I_n$ where $n\in \N$.  We will assume that the intervals are
numbered in the natural order, so that, if $i_n$ is the left endpoint of $I_n$ then
$i_0 = 0$ and $I_n = [i_n , i_{n+1} )$. We say that the interval partition $\{I_n : n \in  \N\}$
dominates another interval partition $\{J_n : n \in \N\}$ if for all but finitely many $n\in \N$
there is $k\in \N$ such that $J_k\subseteq I_n$.   By Theorem 2.10 of \cite{blass} the
 dominating number $\mathfrak d$ is equal to the smallest cardinality 
 of a family of interval partitions dominating 
 all interval partitions. By a $c_0$-matrix we will mean a matrix of a linear
 operator on $c_0$ in the sense of Lemma \ref{operatorsc0}.

\begin{proposition}\label{rosc0-d}
$\mathfrak{ros}({c_0})\leq \mathfrak{d}$.
\end{proposition}
\begin{proof}  
Note that there is a family $\mathcal A=\{A_\alpha: \alpha<\mathfrak d\}$ of
infinite subsets of $\N$ such that for every function $f:\N\rightarrow\N$ 
there is $\alpha<\mathfrak d$ such that $f(k)<n$ whenever
$k<n$ are two elements of $A_\alpha$. 
To prove this assume that $f$ is strictly increasing and $f(0)=1$ and consider the interval partition
 $\mathcal I=\{[f^{2i}(0), f^{2i+2}(0)): i\in \N\}$.
  If $\mathcal J$ is an interval partition that dominates
$\mathcal I$, then for almost all endpoints $k<n$ of intervals
in $\mathcal J$ there is $i\in \N$ such that  $k\leq f^{2i}(0)< f^{2i+2}(0)\leq n$. 
In particular $f(k)\leq f^{2i+1}(0)< f^{2i+2}(0)\leq n$, so if we take 
as a set $A$ the endpoints of the intervals of $\mathcal J$ minus some finite set,
  we obtain that $f(k)<n$ whenever
$k<n$ are two elements of $A$. So, as $\mathcal A=\{A_\alpha: \alpha<\mathfrak d\}$ we take the 
family of all finite modifications of the sets of all the endpoints of partitions
from a family of interval partitions of cardinality $\mathfrak d$
which is  dominating  and  which
exists by the discussion at the beginning of this Section.

We will see  that for each $\varepsilon>0$  each  $c_0$-matrix  $M=(m_{k, n})_{k, n\in \N}$
  is $\varepsilon$-fragmented by an element of
$\mathcal{A}$.  Let $\|M\|_\infty=\rho$.
Find a function $f_M : \N \rightarrow \N$ such that  for every 
$n\in \N\setminus\{0\}$ we have
     $$m_{k,n} \leq \frac{\varepsilon}{2^{n+2}}\leqno (1)$$
    for all $k \geq f_M(n)$. Its  existence follows from the fact
    that $M$ is a $c_0$-matrix, and so
    $(m_{k, n})_{k\in \N}$ converges to $0$ for each $n\in \N$.
    Now find  a function  $g_M : \N \rightarrow \N$ 
     such that for any $k \in \N$, 
$$\displaystyle \sum_{n \geq g_M(k)} m_{k,n} \leq \frac{\varepsilon}{2}.\leqno (2)$$ 
Its  existence follows from the fact
    that $\sum_{n\in \N}m_{k, n}\leq\rho$ for every $k\in \N$.
    Now find $\alpha< \mathfrak{d}$ such that 
for any $i<j$ in $A_\alpha$ we have $\max(f_M(i), g_M(i))<j$.

We claim that $M$ is 
        ${\varepsilon}$-fragmented by $A_{\alpha}$. Let $k \in A_{\alpha}$. 
        If $n \in A_{ \alpha}$ and $n < k$, 
        then $f_M(n) < k$, so $m_{k,n} \leq \frac{\varepsilon}{2^{n + 2}}$  by (1) and 
        therefore $\sum_{n < k} m_{k,n} \leq \frac{\varepsilon}{2}$. 
        On the other hand, if $k < n$, then $g_M(k) < n$ and 
        therefore 
$ \sum_{n \in A_{\alpha}, n> k}m_{k,n} \leq \frac{\varepsilon}{2}$ by (2). 
Therefore
$$
\sum_{n \in A_{ \alpha}\setminus\{k\}}m_{k,n} \leq \sum_{n \in A_{\alpha}, n < k} m_{k, n}
 + \sum_{n \in A_{\alpha}, n> k}m_{k, n} \leq {\varepsilon}
$$
which completes the proof.
\end{proof}

\begin{proposition}\label{dr-ros01c0}
$\min\{\mathfrak{d}, \mathfrak{r}\} \leq \mathfrak{ros}_{01}({c_0})$.
\end{proposition}
\begin{proof}  
Let $\kappa < \min\{ \mathfrak{d}, \mathfrak{r} \}$ and let
  $\mathcal A=\{A_\alpha : \alpha \in \kappa \}$ be a family of 
  infinite subsets of $\N$ closed under finite modifications.
We will construct a finite-to-one $f:\N\rightarrow\N$ such that 
$f[A_\alpha]\cap A_\alpha\not=\emptyset$ 
for any  $\alpha<\kappa$.

First, by Theorem \ref{ros01-main},
because $\kappa < \mathfrak{r}=\mathfrak{ros}_{01}$, 
there is  $g:\N\rightarrow\N$ with no fixed points such that 
$g[A_\alpha]\cap A_\alpha\not=\emptyset$ for every $\alpha<\kappa$.   For every $\alpha<\kappa$,
recursively construct an increasing function $f_\alpha$  such that 
for every $n\in\N$ we have 
$$g\big[[f_\alpha (n), f_\alpha (n+1)) \cap
 A_\alpha\big]\cap \big([f_\alpha (n), f_\alpha (n+1)) \cap
 A_\alpha\big)\not=\emptyset.\leqno (3)$$
  This can be done as 
 $g[A_\alpha\setminus [0, f_\alpha(n))]\cap \big(A_\alpha\setminus [0, f_\alpha(n))\big)\not=\emptyset$
  since $\mathcal A$ is closed under finite modifications.
 
 Secondly  because $\kappa < \mathfrak{d}$ by the discussion at the beginning of this section
 there is an interval partition $\mathcal{I}=\{ I_n : n \in \N \}$ which dominates
 each interval partition $\mathcal J^\alpha=\{[f_\alpha(n), f_\alpha(n+1)): n\in \N\}$
 for $\alpha<\kappa$.
  Define $f:\N\rightarrow\N$ as follows:
$$
f(i) = \begin{cases}
g(i) & \text{ if }i,g(i)\text{ are in the same piece of }\mathcal{I},\\
i+1 & \text{otherwise}.
\end{cases}
$$
Clearly $f$ is finite-to-one and with no fixed points. To finish the
 proof note that if $\alpha < \kappa$, then there is an 
 $n \in \N$ such that $C = [f_\alpha (n) , f_\alpha (n + 1)) \cap A_\alpha$ 
 is included in a single piece of $\mathcal{I}$, so if $i,j \in C$, 
 then $f(i)=g(i)$ and by (3) we have 
 $f[A_\alpha]\cap A_\alpha\not=
 \emptyset$.
 
 \end{proof}

 \begin{theorem}\label{rosc0-main}
$\mathfrak{ros}({c_0})=\min(\mathfrak{d}, \mathfrak{r}).$
\end{theorem}
\begin{proof} We have $\mathfrak{ros}({c_0})\leq \mathfrak{ros}=\mathfrak{r}$ by
Proposition \ref{inequalities-operators} (\ref{rosc0-ros}) and by Theorem \ref{main-theorem}. 
So Proposition
\ref{rosc0-d} implies $\mathfrak{ros}({c_0})\leq\min(\mathfrak{d}, \mathfrak{r}).$
The other inequality follows from Proposition \ref{inequalities} (\ref{ros01c0-rosc0})
and Proposition \ref{dr-ros01c0}. 
\end{proof}
 
\begin{theorem}\label{ros01c0-main}
$\mathfrak{ros}_{01}(c_0)=\min(\mathfrak d, \mathfrak{r}).$
\end{theorem}
\begin{proof} Use Proposition \ref{dr-ros01c0}, Theorem \ref{rosc0-main}
and Proposition \ref{inequalities} (\ref{ros01c0-rosc0}).
\end{proof}

We should add here that $\min(\mathfrak{d}, \mathfrak{r})$ is investigated in
\cite{aubrey} where for example it is proved that 
$\min(\mathfrak{d}, \mathfrak{r})=\min(\mathfrak{d}, \mathfrak{u})$.

\bibliographystyle{amsplain}

\end{document}